\documentclass[10pt]{article}

\usepackage[T1]{fontenc}
\usepackage{lmodern}
\usepackage[margin=1in]{geometry}
\usepackage{graphicx}
\usepackage{amsmath,amssymb,amsfonts}
\usepackage{amsthm}
\usepackage{mathtools}
\usepackage{booktabs}
\usepackage{array}
\usepackage{xcolor}
\usepackage{tikz}
\usetikzlibrary{positioning,calc}
\usepackage[
    colorlinks=true,
    linkcolor=blue,
    citecolor=blue,
    urlcolor=blue
]{hyperref}

\theoremstyle{plain}
\newtheorem{theorem}{Theorem}[section]
\newtheorem{lemma}[theorem]{Lemma}
\newtheorem{proposition}[theorem]{Proposition}
\newtheorem{corollary}[theorem]{Corollary}

\theoremstyle{definition}
\newtheorem{definition}{Definition}

\theoremstyle{remark}
\newtheorem{remark}{Remark}

\newcommand{\angles}[1]{\left\langle #1\right\rangle}
\newcommand{\TG}{\mathcal T_G}
\newcommand{\TsubH}{\mathcal T_H}
\newcommand{\WH}{\mathcal W_H}
\newcommand{\Mstar}{M_G^{\ast}}
\newcommand{\mugen}{\mu_G^{\mathrm{gen}}}
\newcommand{\DkappaG}{D_{\kappa_G}}
\newcommand{\NkappaG}{N_{\kappa_G}}
\newcommand{\Dset}{\mathcal D}
\newcommand{\Hset}{\mathcal H}
\newcommand{\Nset}{\mathcal N}
\newcommand{\Fp}{\mathbb F_p}
\newcommand{\Z}{\mathbb Z}

\DeclareMathOperator{\ord}{ord}
\DeclareMathOperator{\StabL}{Stab_L}
\DeclareMathOperator{\StabR}{Stab_R}

\hypersetup{
  pdftitle={Minimal Product Sets in Finite Metacyclic Groups},
  pdfauthor={Santiago Jimenez Ramos},
  pdfkeywords={minimum product sets, finite metacyclic groups, additive combinatorics, critical subset trios, finite groups}
}
\begin{document}

\title{Minimal Product Sets in Finite Metacyclic Groups}
\author{Santiago Jimenez Ramos\\
Centro de Ci\^encias Exatas e da Natureza, Universidade Federal de Pernambuco\\
Recife, Pernambuco, Brazil\\
\texttt{santiago.jimenez@ufpe.br}}
\date{}

\maketitle

\begin{abstract}
Minimum product-set problems ask how small a product of two subsets of prescribed cardinalities can be in a finite group. Although the problem is completely understood for finite abelian groups, exact results in the nonabelian setting are available only for particular classes. We determine the minimum product-set function $\mu_G$ for every finite metacyclic group, without assuming that the relevant cyclic extension splits. Our proof converts the extremal problem into a finite recursive arithmetic description of the feasible cardinality vectors of subset trios. The recursion combines the subgroup structure of finite metacyclic groups with the structural classification of maximal critical subset trios. The resulting characterization recovers the previously known formulas for abelian and nilpotent groups, dihedral groups, and nonabelian groups of order \(pq\). It also yields metacyclic examples in which the minimum product-set size lies strictly between the classical lower and upper arithmetic bounds for finite solvable groups. For faithful split metacyclic groups, we further show that requiring one factor to be symmetric, contain the identity, and generate the whole group can force a strictly larger product set than the unrestricted minimum. The latter constraint admits a vertex-isoperimetric interpretation on connected Cayley graphs.
\end{abstract}


\section{Introduction}\label{sec:introduction}

For subsets $A,B$ of a finite group $G$, write
\(
AB=\{ab:a\in A,\ b\in B\}.
\)
For integers $1\le r,s\le |G|$, define
\[
\mu_G(r,s)
=
\min\bigl\{
|AB|:A,B\subseteq G,\ |A|=r,\ |B|=s
\bigr\}.
\]
The problem is to determine how small a product set can be when the
cardinalities of its two factors are prescribed.  In abelian groups this
question is closely controlled by subgroup orders, while in nonabelian groups
the exact value can depend more delicately on the structure of the ambient
group.

The classical case is the Cauchy--Davenport theorem for cyclic groups of prime
order \cite{Cauchy1813,Davenport1935}.  Eliahou, Kervaire and Plagne
\cite{EliahouKervairePlagne2003} determined $\mu_G$ for every finite abelian
group.  Eliahou and Kervaire later established arithmetic bounds for finite
solvable groups and determined the dihedral case
\cite{EliahouKervaire2007,EliahouKervaire2010}.  Among further nonabelian
classes, Deckelbaum \cite{Deckelbaum2009} studied groups of order $pq$, and
Berchenko--Kogan \cite{BerchenkoKogan2012} completed that case and determined
$\mu_G$ for finite nilpotent groups.  More recently, Benavides and Mutis
\cite{BenavidesMutis2025} treated a family of nonabelian metacyclic groups in
which the cyclic quotient acts by inversion.  The corresponding problem for
general finite metacyclic groups, where the action need not be inversion and
the relevant cyclic extension need not split, remains outside these formulas.

The main objective of this paper is to determine $\mu_G(r,s)$ for every finite
metacyclic group $G$ and every admissible pair $(r,s)$.  We do this by replacing
the original minimization problem by an exact description of the possible
cardinalities of subset trios.  For each finite metacyclic group $G$ we define
recursively a finite set
$
\mathcal T_G\subseteq\{0,\ldots,|G|\}^3
$
and prove that
\[
\mathcal T_G
=
\bigl\{
(|A|,|B|,|C|):
(A,B,C)\text{ is a subset trio in }G
\bigr\}.
\]
Consequently,
\[
\mu_G(r,s)
=
|G|-\max\{t:(r,s,t)\in\mathcal T_G\}.
\]
Thus the determination of the minimum product-set function becomes a finite
recursive calculation on subgroup and quotient data.

The construction uses two structural results.  DeVos's classification of
maximal critical subset trios \cite{DeVos2013} describes the configurations
that may occur when product growth is small, including beats, cyclic chords,
dihedral chords and a finite collection of terminal cases.  Yang’s parametrization of the subgroups of finite metacyclic groups \cite{Yang2020} provides the explicit subgroup and quotient data required by the recursion.  The rules defining $\mathcal T_G$ are obtained
from these configurations, and the proof consists in showing that they are
both realizable and exhaustive.

Besides recovering the previously determined abelian, nilpotent, dihedral and
order-$pq$ cases, the exact formula exhibits behavior not detected by the
classical arithmetic bounds.  In particular, there are finite metacyclic
groups and parameters $r,s$ for which
\[
\DkappaG(r,s)
<
\mu_G(r,s)
<
\NkappaG(r,s),
\]
so the minimum lies strictly between the lower and upper
Eliahou--Kervaire bounds.

The exact determination also allows us to compare unrestricted minimizers with
sets satisfying a natural generation condition.  For the faithful split
family
\[
G=C_p\rtimes_h C_m,
\qquad
\operatorname{ord}_p(h)=m>2,
\]
we prove that, if $m<r$ and $2r-1\le p$, then every pair
$A,B\subseteq G$ with
\[
|A|=|B|=r,\qquad
1\in B,\qquad
B=B^{-1},\qquad
\langle B\rangle=G
\]
satisfies
\(
|AB|\ge 2r.
\)
In the same range, the unrestricted minimum equals $2r-2$ when $m$ and $r$
are even and $2r-1$ otherwise.  Hence the unrestricted minimum cannot be
attained by a symmetric generating second factor.  This constrained problem
has a direct Cayley-graph interpretation: if $S=B\setminus\{1\}$, then
$\operatorname{Cay}(G,S)$ is connected and
\[
|AB|-|A|=|AS\setminus A|,
\]
so the difference is an outer vertex boundary.  This places the generating
constraint within the isoperimetric viewpoint on small product sets developed
by Hamidoune \cite{Hamidoune1996,Hamidoune1999}.

We establish this generation-constrained comparison for the faithful split family in the diagonal range above.

Section~\ref{sec:supporting} collects the structural results used in the
proof.  Section~\ref{sec:region} defines the recursive region $\mathcal T_G$.
Sections~\ref{sec:soundness} and~\ref{sec:completeness} prove its soundness and
completeness, and Section~\ref{sec:main} derives the exact formula for
$\mu_G$.  Section~\ref{sec:special} records special cases and the additional
consequences described above.
\section{Definitions and supporting results}\label{sec:supporting}

Groups in the recursive construction are finite and written multiplicatively; auxiliary additive groups will be specified explicitly.

\subsection{Product-set bounds}

We use two results of Eliahou and Kervaire for finite solvable groups. The first provides the elementary trios used at the base of the recursion.

\begin{theorem}[Eliahou--Kervaire \cite{EliahouKervaire2006}]\label{thm:ek-small}
Let $G$ be a finite solvable group. For every $1\le r,s\le |G|$, there exist subsets $A,B\subseteq G$ such that $|A|=r$, $|B|=s$, and $|AB|\le r+s-1$.
\end{theorem}

For positive integers $r,s,d$, set
\[
f_d(r,s)=d\left(\left\lceil\frac rd\right\rceil+
                 \left\lceil\frac sd\right\rceil-1\right).
\]
For a finite group $G$, write
\[
\Dset(G)=\{d:d\mid |G|\},\qquad
\Hset(G)=\{|H|:H\le G\},\qquad
\Nset(G)=\{|H|:H\triangleleft G\},
\]
and define
\[
\begin{aligned}
\DkappaG(r,s)&=\min_{d\in\Dset(G)}f_d(r,s),\\
\kappa_G(r,s)&=\min_{d\in\Hset(G)}f_d(r,s),\\
\NkappaG(r,s)&=\min_{d\in\Nset(G)}f_d(r,s).
\end{aligned}
\]

We use the following form of \cite[Theorem~6 and Section~5]{EliahouKervaire2010}.

\begin{theorem}[Eliahou--Kervaire]\label{thm:ek-bounds}
If $G$ is a finite solvable group and $1\le r,s\le |G|$, then
\[
\DkappaG(r,s)\le \mu_G(r,s)\le \NkappaG(r,s).
\]
\end{theorem}

Since $\Nset(G)\subseteq\Hset(G)\subseteq\Dset(G)$,
\[
\DkappaG(r,s)\le\kappa_G(r,s)\le\NkappaG(r,s).
\]
For metacyclic groups, the divisor and subgroup-order minima coincide by Lemma~\ref{lem:divisors-subgroups}.

\subsection{Subset trios}

\begin{definition}
A \emph{subset trio} in a finite group $G$ is a triple $\Phi=(A,B,C)$ of subsets of $G$ such that $1\notin ABC$. Its cardinality vector is $(|A|,|B|,|C|)$, and its deficiency is
\[
\delta(\Phi)=|A|+|B|+|C|-|G|.
\]
The trio is called \emph{critical} if $\delta(\Phi)>0$, and \emph{nontrivial} if $A,B,C\ne\varnothing$.
\end{definition}

A subset trio is maximal if none of its components can be enlarged while preserving the condition $1\notin ABC$. Equivalently, maximality means that $C=G\setminus(AB)^{-1}$ and the two cyclically permuted identities hold. Every subset trio extends to a maximal one: among its extensions choose one maximizing the sum of the three cardinalities. This operation preserves nontriviality and can only increase deficiency.

Following DeVos \cite[Section~1.4]{DeVos2013}, two subset trios are similar if one is obtained from the other by a sequence of cyclic permutations, the operation
\[
(A,B,C)\longmapsto(C^{-1},B^{-1},A^{-1}),
\]
and the transformations
\[
(A,B,C)\longmapsto(Ag,g^{-1}B,C),\qquad
(A,B,C)\longmapsto(A,Bg,g^{-1}C),
\]
or
\[
(A,B,C)\longmapsto(g^{-1}A,B,Cg),
\]
for $g\in G$. These operations preserve maximality, criticality and deficiency, and preserve the cardinality vector up to permutation. In particular, the simultaneous change
\[
(A,B,C)\longmapsto(xAy^{-1},yBz^{-1},zCx^{-1})
\]
is a similarity.

\begin{lemma}\label{lem:elementary-trios}
Let $G$ be a finite metacyclic group of order $N$. If $a,b,c$ are positive integers satisfying
\[
a+b+c\le N+1,
\]
then there exists a subset trio in $G$ with cardinality vector $(a,b,c)$.
\end{lemma}

\begin{proof}
A metacyclic group is solvable, since it has an abelian normal subgroup with abelian quotient. By Theorem~\ref{thm:ek-small}, choose $A,B$ of sizes $a,b$ with $|AB|\le a+b-1$. The complement $G\setminus(AB)^{-1}$ has at least $N-a-b+1\ge c$ elements. Any $c$-element subset $C$ of this complement gives the required trio.
\end{proof}

\begin{lemma}[Pair--trio bridge]\label{lem:pair-trio}
Let $A,B\subseteq G$ be nonempty and suppose $AB\ne G$. Set $C=G\setminus(AB)^{-1}$. Then $(A,B,C)$ is a nontrivial subset trio and
\[
\delta(A,B,C)=|A|+|B|-|AB|.
\]
Moreover, if $(A,B,C)$ is extended to a maximal subset trio $(A^+,B^+,C^+)$, then $A^+B^+=AB$ and $C^+=C$.
\end{lemma}

\begin{proof}
By construction, $C^{-1}=G\setminus AB$, so $AB\cap C^{-1}=\varnothing$ and
\[
\delta(A,B,C)=|A|+|B|-|AB|.
\]
If $(A^+,B^+,C^+)$ is a maximal extension, then $C^+\supseteq C$. Product-freeness therefore gives
\[
A^+B^+\subseteq G\setminus C^{-1}=AB.
\]
The reverse inclusion follows from $A\subseteq A^+$ and $B\subseteq B^+$. Hence $A^+B^+=AB$, and maximality gives $C^+=G\setminus(A^+B^+)^{-1}=C$.
\end{proof}

\begin{lemma}\label{lem:symmetry}
The set of cardinality vectors of subset trios in $G$ is invariant under permutations of the three coordinates and under coordinatewise decrease.
\end{lemma}

\begin{proof}
Cyclic permutation preserves product-freeness. Moreover,
\[
(A,B,C)\longmapsto(C^{-1},B^{-1},A^{-1})
\]
preserves product-freeness and interchanges the first and third coordinates. These operations generate all permutations of the three coordinates. Finally, replacing any component by a subset preserves product-freeness.
\end{proof}

\subsection{Finite metacyclic groups}

\begin{definition}
A finite group $G$ is called \emph{metacyclic} if it contains a cyclic normal subgroup $N\triangleleft G$ such that $G/N$ is cyclic.
\end{definition}

We use the H\"older presentation in Yang's convention \cite[Section~1]{Yang2020}: the cyclic kernel has order $n$, and the cyclic quotient has order $m$.

\begin{theorem}[Yang]\label{thm:yang-presentation}
Let $G$ be a finite metacyclic group. Then there exist positive integers $m,n$ and integers $g,h$, with
\[
0\le g,h<n,\qquad n\mid g(h-1),\qquad n\mid(h^m-1),
\]
such that
\[
G\cong K(m,n,g,h)=
\angles{\tau,\eta\mid \tau^n=1,\ \eta^m=\tau^g,\ \eta\tau\eta^{-1}=\tau^h}.
\]
Moreover, $|G|=mn$, $\angles{\tau}\triangleleft G$ has order $n$, and $G/\angles{\tau}$ is cyclic of order $m$. Every element of $G$ has a unique expression
\[
\tau^i\eta^j,\qquad 0\le i<n,\quad 0\le j<m.
\]
\end{theorem}

We next record the part of Yang's subgroup parametrization that will be needed in the recursive construction.

\begin{definition}\label{def:gamma-K}
Let $K=K(m,n,g,h)$. Define $\Gamma_K$ to be the set of triples $(k,\ell,\beta)$ satisfying
\[
\ell\mid n,\qquad \ell\mid k,\qquad k\mid m\ell,
\qquad 0\le\beta<\frac n\ell,
\]
and
\[
\beta\sum_{j=0}^{k/\ell-1}h^{(m\ell/k)j}
\equiv -g\pmod{n/\ell}.
\]
For $(k,\ell,\beta)\in\Gamma_K$, set
\[
\Psi(k,\ell,\beta)=
\angles{\tau^\beta\eta^{m\ell/k},\tau^{n/\ell}}.
\]
\end{definition}

\begin{theorem}[Yang \cite{Yang2020}]\label{thm:yang-subgroups}
Let $K=K(m,n,g,h)$. The map
\[
\Psi:\Gamma_K\longrightarrow\{H:H\le K\}
\]
is a bijection. Moreover, if
\[
H=\Psi(k,\ell,\beta),
\]
then
\[
|H|=k,\qquad |H\cap\angles{\tau}|=\ell.
\]
Set
\[
q=\frac{k}{\ell},\qquad
D=\frac{n}{\ell},\qquad
t=\frac{m\ell}{k},\qquad
S_t=\sum_{j=0}^{q-1}h^{tj}.
\]
Then
\[
H\cong K\left(q,\ell,
\frac{g+\beta S_t}{D}\bmod\ell,\ h^t\bmod\ell\right).
\]
Furthermore, $H\triangleleft K$ if and only if
\[
\beta(h-1)\equiv0\pmod D,
\qquad
h^t\equiv1\pmod D.
\]
In this case,
\[
K/H\cong K(t,D,-\beta\bmod D,h\bmod D).
\]
\end{theorem}

\begin{proof}
The parametrization, orders and normality criterion are Yang's Lemma~3.2 and Theorems~3.2--3.3. For the presentation formulas, put $u=\tau^\beta\eta^t$ and $v=\tau^D$. Then $v$ has order $\ell$, $uvu^{-1}=v^{h^t}$ and
\[
u^q=\tau^{g+\beta S_t}=v^{(g+\beta S_t)/D}.
\]
The resulting presentation has order $q\ell=|H|$, so its natural map onto $H$ is an isomorphism. If $H$ is normal, the images of $\tau,\eta$ in $K/H$ satisfy $\tau^D=1$ and $\eta^t=\tau^{-\beta}$, with the original action. Their presentation has order $tD=|K/H|$, proving the quotient formula. Reducing $g+\beta S_t$ modulo $n$ before dividing by $D$ gives the same parameter modulo $\ell$.
\end{proof}

\begin{lemma}\label{lem:divisors-subgroups}
Every divisor of the order of a finite metacyclic group $G$ is a subgroup order. Consequently,
\[
\DkappaG(r,s)=\kappa_G(r,s)
\]
for all valid $r,s$.
\end{lemma}

\begin{proof}
Write $G=K(m,n,g,h)$ and let $k\mid mn$. Put $\ell=\gcd(k,n)$, $q=k/\ell$ and $D=n/\ell$. Comparing prime valuations shows that $q\mid m$ and $\gcd(q,D)=1$. With $t=m/q$ and $S_t=\sum_{j=0}^{q-1}h^{tj}$, the consistency relation $g(h-1)\equiv0\pmod n$ gives $gS_t\equiv qg\pmod D$. If $D>1$, choose $\beta\equiv-q^{-1}g\pmod D$; then $\beta S_t\equiv-g\pmod D$. If $D=1$, take $\beta=0$. Thus $(k,\ell,\beta)$ satisfies Definition~\ref{def:gamma-K}, and Yang's subgroup $\Psi(k,\ell,\beta)$ has order $k$. The equality of the two minima follows.
\end{proof}

\begin{proposition}\label{prop:closure}
Every subgroup and every quotient of a finite metacyclic group is metacyclic.
\end{proposition}

\begin{proof}
Let $G$ be metacyclic and choose a cyclic normal subgroup $N\triangleleft G$ such that $G/N$ is cyclic. If $H\le G$, then $H\cap N$ is cyclic and normal in $H$, while
\[
H/(H\cap N)\cong HN/N
\]
is a subgroup of the cyclic group $G/N$. Hence $H$ is metacyclic. If $L\triangleleft G$, then $NL/L$ is cyclic and normal in $G/L$, and
\[
(G/L)/(NL/L)\cong G/NL
\]
is a quotient of the cyclic group $G/N$. Hence $G/L$ is metacyclic.
\end{proof}

\subsection{Critical subset trios}

For a subset $A\subseteq G$, write
\[
\StabL(A)=\{g\in G:gA=A\},
\qquad
\StabR(A)=\{g\in G:Ag=A\}.
\]
The proof of the main theorem uses the classification of maximal critical subset trios obtained by DeVos \cite[Theorem~2.3]{DeVos2013}. We first record an elementary restriction that applies to the metacyclic setting.

\begin{lemma}\label{lem:cyclic-commutator}
Let $G$ be a finite metacyclic group. Then $G'$ is cyclic. In particular, every section of $G$ has cyclic commutator subgroup.
\end{lemma}

\begin{proof}
Let $N\triangleleft G$ be cyclic with $G/N$ cyclic. Since $G/N$ is abelian, $G'\le N$. Hence $G'$ is cyclic. By Proposition~\ref{prop:closure}, every subgroup and every quotient of $G$ is metacyclic, so the same argument applies to every section of $G$.
\end{proof}

The following is the specialization of DeVos's classification that will be used throughout the proof.

\begin{proposition}[Metacyclic specialization of DeVos's classification]\label{prop:devos-specialization}
Let $\Phi$ be a nontrivial maximal critical subset trio in a finite metacyclic group $G$. Then there exist subgroups
\[
G=G_1>G_2>\cdots>G_m
\]
and subset trios $\Phi=\Phi_1,\Phi_2,\ldots,\Phi_m$ such that:
\begin{enumerate}
\item $\Phi_i$ is a subset trio in $G_i$ for $1\le i\le m$;
\item for $1\le i<m$, the trio $\Phi_i$ is either an impure beat, an impure cyclic chord, or an impure dihedral chord of type $1$ or $2A$, with continuation $\Phi_{i+1}$;
\item the terminal trio $\Phi_m$ is either a pure beat, a pure cyclic chord, a pure dihedral chord, an impure dihedral chord of type $0$ or $2B$, or, after replacing $\Phi_m$ by a similar trio $(A,B,C)$, there exist a possibly trivial normal subgroup $H\triangleleft G_m$ and subgroups $H_1,H_2,H_3\le G_m$ with $H\le H_i$ and $H\triangleleft H_i$ for $i=1,2,3$ such that
\[
\begin{aligned}
H_1&=\StabL(A)=\StabR(C),\\
H_2&=\StabR(A)=\StabL(B),\\
H_3&=\StabR(B)=\StabL(C).
\end{aligned}
\]
and one of the following holds:
\begin{enumerate}
\item[(a)]
\[
|A|=2|H_2|=|B|,
\qquad
\delta(A,B,C)=|H_1|=|H_3|<|H_2|;
\]
\item[(b)]
\[
9|H_1|=|A|=3|H_2|,\qquad
2|H_2|=|B|=3|H_3|,
\]
\[
\delta(A,B,C)=|H_1|,
\qquad
[H_1:H]\mid16;
\]
\item[(c)]
\[
4|H_1|=|A|=2|H_2|,\qquad
3|H_2|=|B|=2|H_3|,
\]
\[
\delta(A,B,C)=|H_1|,
\qquad
[H_1:H]\mid16.
\]
\end{enumerate}
\end{enumerate}
\end{proposition}

\begin{proof}
DeVos's classification \cite[Theorem~2.3]{DeVos2013} gives the stated chain and alternatives, together with one additional terminal case, alternative $3(d)$. In that case, $G_m/H$ is isomorphic to one of
\[
C_2\times S_4,\quad C_2\times A_4,\quad S_4,\quad C_2\times A_5,
\quad A_5,\quad S_5,\quad A_6,\quad S_6.
\]
Each of these groups contains a subgroup isomorphic to $A_4$. Since $A_4'\cong C_2\times C_2$, its commutator subgroup is noncyclic; as the commutator subgroup of a subgroup is contained in that of the ambient group, the displayed groups also have noncyclic commutator subgroup. On the other hand, $G_m/H$ is a section of the metacyclic group $G$, and hence has cyclic commutator subgroup by Lemma~\ref{lem:cyclic-commutator}. This is a contradiction.
\end{proof}

\section{The exact arithmetic trio region}\label{sec:region}

Let $G$ be a finite metacyclic group of order $N$. Define recursively on $N$ a finite region
\[
\TG\subseteq\{0,1,\ldots,N\}^3.
\]
Its vectors are designed to record the cardinalities of subset trios in $G$; Sections~\ref{sec:soundness} and~\ref{sec:completeness} prove that the correspondence is exact.

We begin with all vectors having a zero coordinate and all positive vectors satisfying
\[
a+b+c\le N+1.
\]
The latter are feasible by Lemma~\ref{lem:elementary-trios}. After adding the vectors described below, we close $\TG$ under permutations of the three coordinates and under coordinatewise decrease.

\subsection{Beat rules}

Let $H<G$ be a proper subgroup, and write $D=|H|$, $q=[G:H]=N/D$. For $1\le j\le q-1$, add
\begin{equation}\label{eq:B0}
B0:\quad \bigl(D,jD,(q-j)D\bigr).
\end{equation}
Moreover, whenever $(a,b,c)\in\TsubH$ has all three coordinates positive, and $0\le j\le q-1$, add
\begin{equation}\label{eq:B1}
B1:\quad \bigl(a,jD+b,(q-j-1)D+c\bigr).
\end{equation}
Rule $B0$ corresponds to the pure beat configuration, whereas $B1$ incorporates a critical continuation inside the proper subgroup $H$.

\subsection{Cyclic-extension rules}

Let $H\triangleleft G$ be a proper normal subgroup such that $G/H\cong C_q$, and put $D=|H|$. For integers $k,\ell\ge1$ satisfying $k+\ell\le q$, add
\begin{equation}\label{eq:C0}
C0:\quad \bigl(kD,\ell D,(q-k-\ell+1)D\bigr).
\end{equation}
If $(a,b,c)\in\TsubH$ has positive coordinates, then for $k,\ell\ge0$ satisfying $k+\ell\le q-1$, add
\begin{equation}\label{eq:C1}
C1:\quad \bigl(kD+a,\ell D+b,(q-k-\ell-1)D+c\bigr).
\end{equation}
The rule $C0$ corresponds to a pure cyclic chord, while $C1$ places a recursively feasible continuation in the boundary cosets of the cyclic quotient.

\subsection{Normal quotient lifting}

Let $H\triangleleft G$ satisfy $1<|H|<N$, and put $D=|H|$. If $(a,b,c)\in\mathcal T_{G/H}$, then add
\begin{equation}\label{eq:Q}
Q:\quad (Da,Db,Dc).
\end{equation}

\subsection{Exceptional bounded-coset rules}

For every proper subgroup $H<G$, write $D=|H|$. For every
\[
(i,j)\in\{(2,2),(2,3),(3,2)\}
\]
and every choice of distinct left cosets $x_1H,\ldots,x_iH$ and distinct right cosets $Hy_1,\ldots,Hy_j$, set
\begin{equation}\label{eq:V-z}
z=z(H;x_1,\ldots,x_i;y_1,\ldots,y_j)
=\left|\bigcup_{\alpha=1}^{i}\bigcup_{\beta=1}^{j}x_\alpha H y_\beta\right|
\end{equation}
and add the vector
\begin{equation}\label{eq:V}
V:\quad (iD,jD,N-z).
\end{equation}
All admissible choices are included; if the required number of distinct cosets is unavailable, that parameter choice adds no vector. Independent left translation of the first component and right translation of the second preserve product cardinality, so it suffices to enumerate choices with $x_1=y_1=1$.

\begin{remark}
The value of $z$ depends on the chosen cosets, not only on $i,j,D$. After the normalization $x_1=y_1=1$, rule $V$ requires at most three free representatives and at most six double translates of $H$.
\end{remark}

\subsection{Dihedral rules}

Suppose that $H\triangleleft G$ and $G/H\cong D_n$, where $D_n$ denotes the dihedral group of order $2n$, $n\ge3$. Let $D=|H|$. For $k,\ell\ge0$ satisfying $k+\ell\le n-2$, add
\begin{equation}\label{eq:D0}
D0:\quad
\bigl(2(k+1)D,2(\ell+1)D,2(n-k-\ell-1)D\bigr).
\end{equation}
For $n\ge4$, define a finite boundary region $\WH\subseteq\{0,\ldots,4D\}^3$ as follows. If $(x,y,z)\in\WH$, then for $k,\ell\ge1$ with $k+\ell\le n-2$, add
\begin{equation}\label{eq:D1}
D1:\quad
\bigl(2(k-1)D+x,2(\ell-1)D+y,2(n-k-\ell-1)D+z\bigr).
\end{equation}

\subsubsection{The dihedral boundary region}

Consider the six vertices $X_0,X_1,Y_0,Y_1,Z_0,Z_1$ and the oriented edges
\[
A_{uv}:X_u\to Y_v,\qquad
B_{vw}:Y_v\to Z_w,\qquad
C_{wu}:Z_w\to X_u,
\]
where $u,v,w\in\{0,1\}$. The underlying graph is $K_{2,2,2}$, with eight oriented triangular faces $(A_{uv},B_{vw},C_{wu})$.

At this stage an edge label is formal: it is empty, full, or proper, with cardinality $0$, $D$, or the specified positive integer. The admissible masks below define $\WH$ by summing the four edge cardinalities in each of the $A$, $B$, and $C$ families and then closing under coordinate permutations. Lemma~\ref{lem:boundary-soundness} realizes every admissible mask simultaneously by subsets of $H$.

\paragraph{W0.}
Every edge is formally empty or full, and every triangular face contains an empty edge.

\paragraph{W1.}
Choose one triangular face and assign to its three edges a vector $(a,b,c)\in\TsubH$ with $1\le a,b,c<D$. Every remaining edge is empty or full, and every other triangular face must contain an empty edge. All masks satisfying this condition are included. The restriction $1\le a,b,c<D$ is automatic for a nontrivial face with no empty edge, since a full component $H$ cannot occur in a subset trio whose other two components are nonempty.

For both patterns with two designated faces, their positions, after a permutation of vertex parts and an interchange of the two vertices within a part, are fixed by
\[
(A_{00},B_{00},C_{00},B_{01},C_{10})=(P,Q,R,S,T).
\]
Thus the designated faces are $(P,Q,R)$ and $(P,S,T)$, sharing the same label $P$. Every other edge is empty or full, and each of the other six faces must contain an empty edge. Every such mask is used. Reversing orientations is accompanied by inversion of labels when these formal patterns are realized.

\paragraph{W2A.}
Choose proper subgroups $K\le L<H$, and write
\[
k_0=|K|,\qquad \ell_0=|L|,\qquad D=|H|.
\]
Let $(a,b,c)\in\mathcal T_K$ satisfy $1\le a,b,c<k_0$, and choose integers
\[
0\le u\le \frac{D}{k_0}-1,
\qquad
1\le v\le \frac{D}{\ell_0}-1.
\]
The five proper boundary labels have cardinalities
\begin{equation}\label{eq:W2A}
(|P|,|Q|,|R|,|S|,|T|)
=\bigl(a,uk_0+b,D-(u+1)k_0+c,v\ell_0,D-v\ell_0\bigr).
\end{equation}

\paragraph{W2B.}
Choose proper subgroups $K,L<H$, write
\[
k_0=|K|,\qquad \ell_0=|L|,
\]
and choose
\[
1\le u\le \frac{D}{k_0}-1,
\qquad
1\le v\le \frac{D}{\ell_0}-1.
\]
The corresponding five proper boundary labels have cardinalities
\begin{equation}\label{eq:W2B}
(|P|,|Q|,|R|,|S|,|T|)
=\bigl(|K\cap L|,uk_0,D-uk_0,v\ell_0,D-v\ell_0\bigr).
\end{equation}

\begin{remark}
The graph $K_{2,2,2}$ is the boundary structure attached to an impure dihedral chord: its six vertices encode the two endpoint positions in the three vertex classes, its twelve edges encode boundary cosets, and its eight triangular faces encode the possible local identity products.
\end{remark}

For every arithmetically allowed boundary pattern, let $x,y,z$ be the sums of the four edge cardinalities in the $XY$, $YZ$, and $ZX$ families, respectively. Include $(x,y,z)$ and all its coordinate permutations in $\WH$.

\subsection{The recursive product-set function}

\begin{definition}\label{def:Mstar}
Let $G$ be a finite metacyclic group of order $N$. Define
\[
\Mstar(r,s)=N-\max\{t:(r,s,t)\in\TG\}.
\]
\end{definition}

Since every recursive reference in the construction of $\TG$ occurs in a proper subgroup or a proper quotient of $G$, and since these groups are again metacyclic by Proposition~\ref{prop:closure}, the definition is recursive on the group order. All remaining parameter ranges are finite. At the trivial group the region consists exactly of the coordinate planes. The inclusion of $(r,s,0)$ ensures that the maximum in Definition~\ref{def:Mstar} is nonempty. Hence $\TG$ and $\Mstar(r,s)$ are well defined.

These rules can be evaluated directly from the H\"older parameters. In the normal form $\tau^x\eta^i$, multiplication is
\begin{equation}\label{eq:holder-multiplication}
(x,i)(y,j)=
\left(x+h^i y+g\left\lfloor\frac{i+j}{m}\right\rfloor\pmod n,
      i+j\pmod m\right).
\end{equation}
For $H=\Psi(k,\ell,\beta)$ and the parameters $D,t$ in Theorem~\ref{thm:yang-subgroups}, membership of a canonical word is the arithmetic condition
\begin{equation}\label{eq:yang-membership}
\tau^x\eta^i\in H
\quad\Longleftrightarrow\quad
t\mid i\ \text{ and }\
x\equiv\beta\sum_{j=0}^{i/t-1}h^{tj}\pmod D.
\end{equation}
The sum is zero when $i=0$. Equation~\eqref{eq:yang-membership} gives containment and intersection tests for the Yang subgroups used in $W2A$ and $W2B$. It also evaluates \eqref{eq:V-z}: for each canonical element $z_0$, one checks whether $x_\alpha^{-1}z_0y_\beta^{-1}\in H$ for one of the at most six relevant pairs $(\alpha,\beta)$. Theorem~\ref{thm:yang-subgroups} supplies the subgroup and quotient presentations. A quotient of order $q$ is cyclic exactly when it contains an element of order $q$; a quotient of order $2n$, $n\ge3$, is dihedral exactly when it contains $R,F$ with
\[
\ord(R)=n,\qquad F^2=1,\qquad F\notin\angles{R},\qquad FRF=R^{-1}.
\]
Together with \eqref{eq:holder-multiplication}, these are finite arithmetic tests from the H\"older parameters and make every rule in the recursion effective.

\section{Soundness of the arithmetic region}\label{sec:soundness}

We prove that every vector generated in $\TG$ is realized by a subset trio in $G$.

\begin{lemma}\label{lem:basic-soundness}
Every vector generated by the rules $B0$, $B1$, $C0$, $C1$, $Q$, and $V$ is the cardinality vector of a subset trio, provided that every recursive input vector is feasible.
\end{lemma}

\begin{proof}
We consider the rules separately. For $B0$, let $H<G$, let $D=|H|$, and let $q=[G:H]$. Choose distinct right cosets $Hx_1,\ldots,Hx_j$ and set
\[
A=H,\qquad
B=\bigcup_{i=1}^{j}Hx_i,\qquad
C=G\setminus(AB)^{-1}.
\]
Since $AB=B$, we have
\[
|A|=D,\qquad |B|=jD,\qquad |C|=(q-j)D.
\]
Thus $B0$ is feasible.

For $B1$, let $(A_0,B_0,C_0)$ be a subset trio in $H$ with cardinality vector $(a,b,c)$. Choose $j$ right cosets $Hx_1,\ldots,Hx_j$, distinct from $H$, and let
\[
A=A_0,\qquad
B=B_0\cup\bigcup_{i=1}^{j}Hx_i.
\]
Choose $q-j-1$ of the remaining right cosets of $H$, say $Hy_1,\ldots,Hy_{q-j-1}$, and put
\[
C=C_0\cup\bigcup_{i=1}^{q-j-1}y_i^{-1}H.
\]
Since $A_0$ is a nonempty subset of $H$, one has $A_0Hx_i=Hx_i$ and
\[
AB=A_0B_0\cup\bigcup_{i=1}^{j}Hx_i.
\]
On the other hand,
\[
C^{-1}=C_0^{-1}\cup\bigcup_{i=1}^{q-j-1}Hy_i.
\]
The two displayed unions are disjoint: inside $H$ this is the subset-trio condition for $(A_0,B_0,C_0)$, and outside $H$ it follows from the distinct right cosets selected. Only right cosets are used, so normality of $H$ is not required. Hence $(A,B,C)$ is product-free and has cardinality vector
\[
\bigl(a,jD+b,(q-j-1)D+c\bigr).
\]

We next consider the cyclic rules. Suppose that $H\triangleleft G$ and $G/H\cong C_q$. For $C0$, let $\pi:G\to G/H$ be the quotient map and identify $G/H$ with $\Z/q\Z$. Take
\[
I=\{0,\ldots,k-1\},\qquad
J=\{0,\ldots,\ell-1\},\qquad
K=(\Z/q\Z)\setminus-(I+J).
\]
Then $|I|=k$, $|J|=\ell$, $|K|=q-k-\ell+1$, and $0\notin I+J+K$. Hence
\[
\bigl(\pi^{-1}(I),\pi^{-1}(J),\pi^{-1}(K)\bigr)
\]
is product-free and realizes $C0$.

For $C1$, choose $x\in G$ such that $xH$ generates $G/H$, and let $(A_0,B_0,C_0)$ be a subset trio in $H$ of sizes $(a,b,c)$. Take the full cosets $x^iH$ for $0\le i<k$ in the first component and $x^jH$ for $0\le j<\ell$ in the second. In the third take the full cosets $Hx^{-t}$ for $k+\ell+1\le t\le q-1$. Empty index ranges contribute nothing. At the three boundary positions use
\[
x^k(x^\ell A_0x^{-\ell}),\qquad
x^\ell B_0,\qquad
C_0x^{-(k+\ell)},
\]
respectively. At the quotient level, every possible identity product is excluded except the one involving the three boundary cosets. For boundary elements we have
\[
x^k(x^\ell a x^{-\ell})x^\ell bc x^{-(k+\ell)}
=x^{k+\ell}abcx^{-(k+\ell)},
\]
which cannot be the identity because $1\notin A_0B_0C_0$. Thus $C1$ is feasible and has cardinality vector
\[
\bigl(kD+a,\ell D+b,(q-k-\ell-1)D+c\bigr).
\]

For the quotient rule $Q$, let $\pi:G\to G/H$ be the natural map and let $(A_0,B_0,C_0)$ be product-free in $G/H$. Then
\[
\bigl(\pi^{-1}(A_0),\pi^{-1}(B_0),\pi^{-1}(C_0)\bigr)
\]
is product-free in $G$. Since every fiber of $\pi$ has cardinality $D=|H|$, its cardinality vector is $(Da,Db,Dc)$.

Finally, consider $V$. With the notation of \eqref{eq:V-z}, define $C=G\setminus(AB)^{-1}$ where
\[
A=\bigcup_{\alpha=1}^{i}x_\alpha H,
\qquad
B=\bigcup_{\beta=1}^{j}Hy_\beta.
\]
Then
\[
AB=\bigcup_{\alpha=1}^{i}\bigcup_{\beta=1}^{j}x_\alpha H y_\beta,
\]
so that $|AB|=z$. Therefore
\[
(|A|,|B|,|C|)=(iD,jD,N-z),
\]
and $V$ is feasible.
\end{proof}

\begin{lemma}[Boundary soundness]\label{lem:boundary-soundness}
Every vector in $\WH$ is realized by a labelling of the twelve edges of $K_{2,2,2}$ by subsets of $H$ such that the three labels on every oriented triangular face form a subset trio, provided that all recursive input vectors are feasible.
\end{lemma}

\begin{proof}
For a pattern of type $W0$, every triangular face contains an empty edge, so the assertion is immediate.

For $W1$, place a subset trio realizing the prescribed vector $(a,b,c)\in\TsubH$ on the three edges of the designated face. All remaining edges are either empty or equal to $H$, and the pattern is chosen so that every other face contains an empty edge. Hence every face is product-free.

Consider $W2A$. The two faces not forced to contain an empty edge share the label $P$. By construction,
\[
(|P|,|Q|,|R|)=\bigl(a,uk_0+b,D-(u+1)k_0+c\bigr),
\]
which is a beat vector obtained from a feasible realization of the recursive input trio $(a,b,c)\in\mathcal T_K$. By Lemma~\ref{lem:basic-soundness}, the labels $P,Q,R$ can therefore be chosen so that $1\notin PQR$. Moreover, $P\subseteq K\le L$. Choose $S$ as a union of $v$ right $L$-cosets and put $T=H\setminus S^{-1}$. Then $PS\subseteq S$, and consequently $1\notin PST$. Thus both nontrivial faces are product-free.

For $W2B$, put $P=K\cap L$. Choose $Q$ as a union of $u$ right $K$-cosets and set $R=H\setminus Q^{-1}$. Since $P\subseteq K$, $PQ\subseteq Q$, and hence $1\notin PQR$. Similarly, choose $S$ as a union of $v$ right $L$-cosets and put $T=H\setminus S^{-1}$. Since $P\subseteq L$, we have $PS\subseteq S$, and therefore $1\notin PST$. Every remaining face in $W2A$ and $W2B$ contains an empty edge. Hence all eight triangular faces are product-free.
\end{proof}

\begin{lemma}[Dihedral lifting]\label{lem:dihedral-lifting}
Every vector generated by $D0$ or $D1$ is feasible, provided that all recursive input vectors are feasible. This remains true when the extension
\[
1\longrightarrow H\longrightarrow G\longrightarrow D_n\longrightarrow1
\]
does not split.
\end{lemma}

\begin{proof}
Let $\pi:G\to D_n$ be the quotient map, and write
\[
D_n=\angles{r,f\mid r^n=f^2=1,\ frf=r^{-1}}.
\]
In the quotient put
\[
A^+=\{1,r,\ldots,r^k\}\{1,f\},
\qquad
B^+=\{1,f\}\{1,r,\ldots,r^\ell\}.
\]
Their product is
\[
X=\{1,r,\ldots,r^{k+\ell}\}\{1,fr^\ell\}.
\]
Since $k+\ell\le n-2$, the three sizes are $2(k+1)$, $2(\ell+1)$ and $2(k+\ell+1)$. The inverse images of $A^+$, $B^+$ and $D_n\setminus X^{-1}$ therefore realize $D0$.

For $D1$, assume $k,\ell\ge1$ and instead put
\[
C^+=(D_n\setminus X^{-1})\cup
\{1,fr^{-k},r^{-k-\ell},fr^\ell\}.
\]
The four added points are distinct elements of $X^{-1}$, so
\[
|C^+|=2(n-k-\ell+1).
\]
The endpoint cosets in the three hulls are organized by the octahedral configuration shown in Figure~\ref{fig:dihedral}. For the third edge family the indices are $(w,u)$, in accordance with the definition of the boundary graph.

\begin{figure}[htpb]
\centering

\begin{minipage}[c]{0.43\textwidth}
\centering

\begin{tikzpicture}[
    scale=1.0,
    edge/.style={
        line width=0.8pt
    },
    vertex/.style={
        circle,
        draw,
        fill=white,
        inner sep=0pt,
        minimum size=4.6pt
    },
    lab/.style={
        font=\small,
        fill=white,
        inner sep=1pt,
        align=center
    }
]


\coordinate (X0) at ( 0.00,  2.05);
\coordinate (X1) at ( 0.00, -2.05);

\coordinate (Y0) at (-2.05,  0.00);
\coordinate (Y1) at ( 2.05,  0.00);

\coordinate (Z0) at (-0.55,  0.78);
\coordinate (Z1) at ( 0.55, -0.78);


\draw[edge] (X0)--(Y0);
\draw[edge] (X0)--(Y1);
\draw[edge] (X1)--(Y0);
\draw[edge] (X1)--(Y1);

\draw[edge] (X0)--(Z0);
\draw[edge] (X0)--(Z1);
\draw[edge] (X1)--(Z0);
\draw[edge] (X1)--(Z1);

\draw[edge] (Y0)--(Z0);
\draw[edge] (Y0)--(Z1);
\draw[edge] (Y1)--(Z0);
\draw[edge] (Y1)--(Z1);


\node[vertex] at (X0) {};
\node[vertex] at (X1) {};
\node[vertex] at (Y0) {};
\node[vertex] at (Y1) {};
\node[vertex] at (Z0) {};
\node[vertex] at (Z1) {};


\node[lab, above=3pt] at (X0)
    {$X_0$\\[-1pt]$(1)$};

\node[lab, below=3pt] at (X1)
    {$X_1$\\[-1pt]$(fr^{-k})$};

\node[lab, left=6pt] at (Y0)
    {$Y_0$\\[-1pt]$(1)$};

\node[lab, right=6pt] at (Y1)
    {$Y_1$\\[-1pt]$(f)$};

\node[lab, above left=1pt and 2pt] at (Z0)
    {$Z_0$\\[-1pt]$(1)$};

\node[lab, below right=1pt and 2pt] at (Z1)
    {$Z_1$\\[-1pt]$(fr^{\ell})$};

\end{tikzpicture}

\end{minipage}
\hfill
\begin{minipage}[c]{0.53\textwidth}
\centering

\small
\renewcommand{\arraystretch}{1.28}
\setlength{\tabcolsep}{4.5pt}

\begin{tabular}{@{}c|cccc@{}}
\text{edge family}
    & $00$ & $01$ & $10$ & $11$
\\
\hline
$A_{uv}=X_u^{-1}Y_v$
    & $r^0$
    & $f$
    & $fr^{-k}$
    & $r^k$
\\
$B_{vw}=Y_v^{-1}Z_w$
    & $r^0$
    & $fr^{\ell}$
    & $f$
    & $r^{\ell}$
\\
$C_{wu}=Z_w^{-1}X_u$
    & $r^0$
    & $fr^{-k}$
    & $fr^{\ell}$
    & $r^{-k-\ell}$
\end{tabular}

\end{minipage}

\caption{Octahedral endpoint configuration for the dihedral prechord
hulls. Parenthesized expressions indicate the chosen representatives
of the six quotient vertices. The three edge families are oriented
from $X$ to $Y$, from $Y$ to $Z$, and from $Z$ to $X$; the table
records the corresponding quotient elements.}
\label{fig:dihedral}
\end{figure}
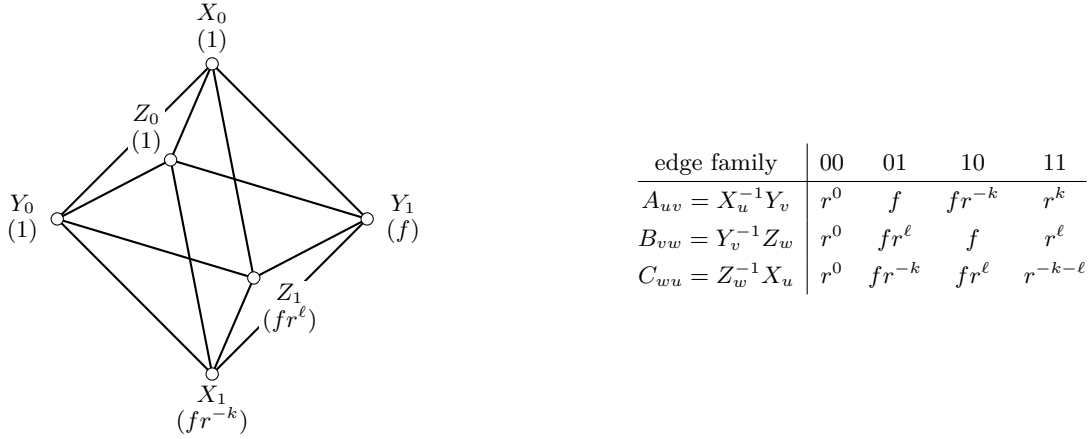
There are exactly eight solutions of $abc=1$ in $A^+\times B^+\times C^+$, namely the triples $(A_{uv},B_{vw},C_{wu})$, with $u,v,w\in\{0,1\}$.

Choose simultaneously six representatives in $G$ of the quotient vertices
\[
\overline X_0=\overline Y_0=\overline Z_0=1,
\qquad
\overline X_1=fr^{-k},
\qquad
\overline Y_1=f,
\qquad
\overline Z_1=fr^\ell.
\]
Every endpoint in the table is $x^{-1}y$ for its oriented edge. Fix a boundary labelling supplied by Lemma~\ref{lem:boundary-soundness}. On an edge from representatives $x$ to $y$ with label $E\subseteq H$, use the subset $x^{-1}Ey$ of the corresponding endpoint coset. Fill every nonendpoint coset of $A^+$, $B^+$, $C^+$ completely. Any identity product would project to one of the eight endpoint triples. On that face the lifted product is
\[
(x^{-1}E_Ay)(y^{-1}E_Bz)(z^{-1}E_Cx)
=x^{-1}(E_AE_BE_C)x,
\]
which does not contain $1$. Thus the constructed components form a subset trio.

There are $2(k-1)$, $2(\ell-1)$ and $2(n-k-\ell-1)$ nonendpoint cosets in the respective hulls. Adding the three boundary totals $(x,y,z)$ gives \eqref{eq:D1}, including the cases $k=1$ or $\ell=1$ when one interior is empty. The cancellation around each face shows that the construction is independent of the chosen lifts and does not require the extension to split.
\end{proof}

\begin{proposition}\label{prop:soundness}
Every vector in $\TG$ is the cardinality vector of a subset trio in $G$. Consequently,
\[
\mu_G(r,s)\le\Mstar(r,s)
\]
for every $1\le r,s\le|G|$.
\end{proposition}

\begin{proof}
We proceed by induction on $|G|$. The vectors with a zero coordinate are trivially feasible, and the positive vectors satisfying $a+b+c\le|G|+1$ are feasible by Lemma~\ref{lem:elementary-trios}. Every recursive occurrence of $\mathcal T_H$ or $\mathcal T_{G/H}$ involves a proper subgroup or a proper quotient, and these are metacyclic by Proposition~\ref{prop:closure}. Hence the inductive hypothesis applies.

Lemma~\ref{lem:basic-soundness} proves the feasibility of the rules $B0$, $B1$, $C0$, $C1$, $Q$, and $V$. Lemma~\ref{lem:boundary-soundness} proves the feasibility of every boundary vector used in $D1$, and Lemma~\ref{lem:dihedral-lifting} proves the rules $D0$ and $D1$.

Finally, coordinate permutations and coordinatewise decrease preserve product-freeness by Lemma~\ref{lem:symmetry}. Therefore every vector in $\TG$ is feasible.

In particular, if $(r,s,t)\in\TG$, then there exist subsets $A,B,C\subseteq G$ with
\[
|A|=r,\qquad |B|=s,\qquad |C|=t,
\]
and $1\notin ABC$. Hence $AB\cap C^{-1}=\varnothing$, so $|AB|\le|G|-t$. Maximizing over all such $t$ gives
\[
\mu_G(r,s)\le |G|-\max\{t:(r,s,t)\in\TG\}=\Mstar(r,s).
\]
\end{proof}

\section{Completeness of the arithmetic region}\label{sec:completeness}

We prove the converse to Proposition~\ref{prop:soundness}: every critical configuration in a finite metacyclic group is captured by the recursion. For $u=(u_1,u_2,u_3)$ and $v=(v_1,v_2,v_3)$, write $u\preceq v$ if $u_i\le v_i$ for $i=1,2,3$.

\begin{lemma}\label{lem:monotonicity}
For fixed subgroup and integer parameters, the maps $B1$, $C1$, $D1$, and the boundary maps $W1$, $W2A$, are coordinatewise monotone in their recursive cardinality input.
\end{lemma}

\begin{proof}
In $B1$ and $C1$, each coordinate is the corresponding input coordinate plus a fixed nonnegative multiple of $|H|$. The same holds for $D1$ as a function of its three boundary totals. For a fixed $W1$ mask each edge-family total is a fixed full-edge contribution plus one coordinate of the designated face. In $W2A$ the shared label $P$ is counted only once, as an edge, although it belongs to two faces. The three family totals therefore contain $a,b,c$, respectively, with coefficient one and with all other summands fixed. Permuting vertex parts only permutes these output coordinates. This proves all the assertions.
\end{proof}

The next three lemmas are stated under the inductive hypothesis that every nontrivial maximal critical subset trio in a proper metacyclic subgroup has its cardinality vector in the corresponding arithmetic region. This hypothesis is closed in Proposition~\ref{prop:completeness}. Downward closure allows a dominating recursive vector to be trimmed to the continuation cardinalities before applying the next rule.

\subsection{Beats and cyclic chords}

\begin{lemma}\label{lem:beat-completeness}
Let $\Phi$ be a maximal critical subset trio occurring as a pure or impure beat in the classification of Proposition~\ref{prop:devos-specialization}. Then its cardinality vector is coordinatewise dominated by a vector of $\TG$.
\end{lemma}

\begin{proof}
For a pure beat, DeVos's definition \cite[Section~1.5]{DeVos2013} gives, after similarity, $A=H$, $HB=B$ and $C=G\setminus B^{-1}$, with $H<G$. Thus $B$ is a union of $j$ right $H$-cosets, and the cardinality vector is $B0$.

For an impure beat, the same definition gives $A\subseteq H$, $H(B\setminus H)=B\setminus H$ and
\[
C\setminus H=\bigl(G\setminus(AB)^{-1}\bigr)\setminus H.
\]
Its continuation is $(A,B\cap H,C\cap H)$, a nontrivial maximal critical trio in $H$. If its sizes are $(a,b,c)$ and $B\setminus H$ consists of $j$ right $H$-cosets, then $AB\setminus H$ consists of exactly the same cosets. Hence the outer sizes are
\[
\bigl(a,jD+b,(q-j-1)D+c\bigr),
\qquad 0\le j\le q-1.
\]
The induction hypothesis places $(a,b,c)$ in $\TsubH$, and $B1$ gives the required vector. This operation preserves deficiency, since the added sizes total $(q-1)D$. Similarity causes no difficulty because coordinate permutations are allowed in $\TG$.
\end{proof}

\begin{lemma}\label{lem:cyclic-completeness}
Let $\Phi$ be a maximal critical subset trio occurring as a pure or impure cyclic chord. Then its cardinality vector is coordinatewise dominated by a vector of $\TG$.
\end{lemma}

\begin{proof}
Use the definitions in \cite[Section~1.5]{DeVos2013}. For a pure cyclic chord, $H$ is normal, $G/H$ is cyclic of order $q$, and all components are $H$-stable. After similarity the first two quotient supports are basic intervals of lengths $k,\ell$ with common generating ratio. Their product has $k+\ell-1<q$ positions and the third component is its inverse complement. Their sizes are precisely $C0$.

For an impure cyclic chord, after similarity the quotient supports of $A,B$ are the intervals $0,\ldots,k$ and $0,\ldots,\ell$, with $k,\ell\ge1$, and all cosets except the one at $0$ are full. The continuation $(A\cap H,B\cap H,C\cap H)$ is nontrivial maximal critical in $H$. Write its sizes as $(a,b,c)$. Every position $1,\ldots,k+\ell$ is full in $AB$. Since $C\cap H\ne\varnothing$, these positions cannot wrap into $H$, and since $C\setminus H\ne\varnothing$, at least one nonidentity quotient position is left over. Thus $k+\ell\le q-2$. Inside $H$, maximality identifies the third continuation component with the inverse complement of $(A\cap H)(B\cap H)$. Outside $H$, the third component occupies the inverses of the $q-k-\ell-1$ unused positions. The resulting sizes are exactly $C1$, which allows this parameter range. Induction on $H$ therefore gives the desired vector, and the deficiency is again that of the continuation.
\end{proof}

\subsection{Dihedral chords}

\begin{lemma}\label{lem:dihedral-completeness}
Let $\Phi$ be a maximal critical subset trio occurring as a pure or impure dihedral chord in Proposition~\ref{prop:devos-specialization}. Then its cardinality vector is coordinatewise dominated by a vector of $\TG$.
\end{lemma}

\begin{proof}
Let $H\triangleleft G$ be the normal subgroup in DeVos's chord definition, put $D=|H|$ and write $G/H\cong D_n$. A pure dihedral chord has full cosets on the three quotient progressions. Their sizes are those of $D0$, with $k,\ell\ge1$ and $k+\ell\le n-2$; the additional restrictions in DeVos's definition only reduce this range.

For an impure chord, \cite[Sections~2.2--2.4 and Lemma~13.9]{DeVos2013} gives the prechord hulls $A^+,B^+,C^+$ used in Lemma~\ref{lem:dihedral-lifting}. Their twelve endpoint cosets produce an octahedral configuration in $H$. The full interiors consequently contribute
\[
2(k-1)D,\qquad 2(\ell-1)D,\qquad 2(n-k-\ell-1)D,
\]
with $k,\ell\ge1$ and $k+\ell\le n-2$. It remains to account for the boundary itself.

In type $0$ all labels are empty or full, giving $W0$. In type $1$ precisely three labels, forming one face, are proper; they constitute the maximal critical continuation in $H$. Induction puts its cardinalities in $\TsubH$, giving $W1$. In either case every other face has an empty label. Indeed, a face with three nonempty labels and a full label $H$ would have product $H$, contradicting the subset-trio condition.

For type $2A$, DeVos's definition supplies a common label $P\subseteq xK\subseteq xL$ with $K\le L<H$. One face is an impure beat relative to the shared first side, whereas enlarging $P$ to $xL$ makes the second face a pure beat. The proof of Lemma~13.8 of \cite{DeVos2013}, using Lemma~13.5(3), identifies the continuation subgroup with a conjugate of the minimal subgroup whose coset contains the shared label. After conjugating the continuation and normalizing the common coset, we may take $P\subseteq K\le L$. The two opposite vertices can be translated independently, so the partial coset of the impure beat may also be placed at $K$ without changing $P$ or any cardinality. If the continuation in $K$ has sizes $(a,b,c)$, its impure face has sizes
\[
\bigl(a,uk_0+b,D-(u+1)k_0+c\bigr),
\qquad 0\le u\le D/k_0-1.
\]
The pure face, before restricting its first component to $P$, has sizes
\[
(\ell_0,v\ell_0,D-v\ell_0),
\qquad 1\le v\le D/\ell_0-1.
\]
These are the five quantities in $W2A$. The continuation is maximal critical in a proper subgroup, so induction certifies its vector, and Lemma~\ref{lem:boundary-soundness} realizes the two faces simultaneously with their common label $P$.

For type $2B$, DeVos's definition gives
\[
P=xK\cap xL=x(K\cap L).
\] The other two components of each face belong to a pure beat relative to $K$ or $L$. After normalizing the common coset, their five sizes are
\[
(|K\cap L|,uk_0,D-uk_0,v\ell_0,D-v\ell_0),
\]
with precisely the ranges in $W2B$. This is a terminal construction and requires no continuation. In types $2A$ and $2B$ every other face contains an empty edge by the same full-label argument. Interchanges of vertices and permutations of vertex parts put the two proper faces in the designated positions of our masks.

These exhaust the impure subset-trio types: maximal impure chords of type $-1$ do not occur, and Lemmas~13.8--13.9 of \cite{DeVos2013} exclude type $2C$ in the subset-trio setting. Hence every boundary vector belongs to $\WH$, and adjoining the full interiors gives exactly rule $D1$.
\end{proof}

\subsection{Exceptional terminal configurations}

\begin{lemma}\label{lem:exceptional-completeness}
Every exceptional terminal configuration occurring in
Proposition~\ref{prop:devos-specialization} is represented by a vector
generated by rule $V$.
\end{lemma}

\begin{proof}
Let $\Gamma$ denote the ambient group of the exceptional terminal trio
(that is, $\Gamma=G_m$ in
Proposition~\ref{prop:devos-specialization}).
Let $(A,B,C)$ be the representative supplied by that proposition, and let
\[
H_2=\StabR(A)=\StabL(B).
\]
Set $D=|H_2|$. Since the trio is nontrivial, $H_2<\Gamma$.
Since $H_2$ is the right stabilizer of $A$, the set $A$ is a union of
left $H_2$-cosets. Similarly, since $H_2$ is the left stabilizer of $B$,
the set $B$ is a union of right $H_2$-cosets.

In alternative (a), $|A|=2D$ and $|B|=2D$. Thus $A$ is a union of two
left $H_2$-cosets and $B$ is a union of two right $H_2$-cosets,
corresponding to the choice $(i,j)=(2,2)$ in rule $V$.

In alternative (b), $|A|=3D$ and $|B|=2D$, so rule $V$ applies with
$(i,j)=(3,2)$.

In alternative (c), $|A|=2D$ and $|B|=3D$, so rule $V$ applies with
$(i,j)=(2,3)$.

After translating the two components, we may assume that one selected
left coset and one selected right coset are $H_2$ itself. Thus we may
write
\[
A=\bigcup_{\alpha=1}^{i}x_\alpha H_2,
\qquad
B=\bigcup_{\beta=1}^{j}H_2y_\beta,
\]
with $x_1=y_1=1$, replacing $C$ by the third component of the
corresponding similar trio. Since the trio is maximal,
\[
C=\Gamma\setminus(AB)^{-1}.
\]
Indeed, product-freeness gives
\[
C\subseteq \Gamma\setminus(AB)^{-1},
\]
and a strict inclusion would allow $C$ to be enlarged, contradicting
maximality. Therefore
\[
|C|=|\Gamma|-|AB|,
\]
where
\[
AB=
\bigcup_{\alpha=1}^{i}
\bigcup_{\beta=1}^{j}
x_\alpha H_2y_\beta.
\]
This is exactly the quantity computed in rule $V$. Hence the
cardinality vector of $(A,B,C)$ belongs to $\mathcal T_\Gamma$, up to
the coordinate permutations allowed in the definition of the region.
\end{proof}

\begin{proposition}[Completeness]\label{prop:completeness}
Let $G$ be a finite metacyclic group. Every nontrivial maximal critical subset trio in $G$ has cardinality vector coordinatewise dominated by a vector of $\TG$.
\end{proposition}

\begin{proof}
Proceed by induction on $|G|$. There is no nontrivial subset trio in the trivial group, so the initial case is empty. For a nontrivial maximal critical trio $\Phi$ in $G$, Proposition~\ref{prop:devos-specialization} gives a chain
\[
\Phi=\Phi_1,\ldots,\Phi_t
\]
through
\[
G=G_1>\cdots>G_t.
\]
All these groups are metacyclic by Proposition~\ref{prop:closure}. The inductive hypothesis applies to every continuation in a group of smaller order, and downward closure gives its actual cardinality vector in that group's region.

At a terminal, a pure beat is supplied by $B0$, a pure cyclic chord by $C0$, a pure dihedral chord by $D0$, a type $0$ or $2B$ dihedral chord by its boundary rule followed by $D1$, and an exceptional terminal by $V$. These assertions are Lemmas~\ref{lem:beat-completeness}--\ref{lem:exceptional-completeness}. Moving backwards along the chain, an impure beat applies $B1$ to its certified continuation and an impure cyclic chord applies $C1$. A type $1$ dihedral chord applies $W1$ and $D1$; a type $2A$ chord applies the nested-subgroup construction $W2A$ and then $D1$. Each map is monotone by Lemma~\ref{lem:monotonicity}, and similarity is absorbed by the coordinate permutations allowed in $\TG$. Propagating backward along the chain yields a vector of $\TG$ dominating the cardinality vector of $\Phi$.
\end{proof}

\section{Main theorem}\label{sec:main}

\begin{theorem}\label{thm:exact-region}
For every finite metacyclic group $G$,
\[
\TG=
\bigl\{(a,b,c):\text{there exists a subset trio in $G$ with cardinality vector }(a,b,c)\bigr\}.
\]
\end{theorem}

\begin{proof}
Proposition~\ref{prop:soundness} gives one inclusion. Conversely, let $(a,b,c)$ be the cardinality vector of a subset trio $\Phi$. If one coordinate is zero, or if $a+b+c\le|G|+1$, then $(a,b,c)$ belongs to $\TG$ by definition.

Otherwise $\Phi$ is nontrivial and critical. Extend it to a maximal subset trio $\Phi^+$. By Proposition~\ref{prop:completeness}, the cardinality vector of $\Phi^+$ is dominated by a vector of $\TG$. Since $\TG$ is closed under coordinatewise decrease, the original vector $(a,b,c)$ also belongs to $\TG$.
\end{proof}

\begin{theorem}[Minimal product sets in finite metacyclic groups]\label{thm:main}
Let $G$ be a finite metacyclic group of order $N$. Then, for every $1\le r,s\le N$,
\[
\mu_G(r,s)=\Mstar(r,s)
=N-\max\{t:(r,s,t)\in\TG\}.
\]
\end{theorem}

\begin{proof}
If $A,B\subseteq G$ satisfy $|A|=r$, $|B|=s$, and $C=G\setminus(AB)^{-1}$, then $(A,B,C)$ is product-free and $|C|=N-|AB|$. Theorem~\ref{thm:exact-region} therefore gives
\[
N-|AB|\le\max\{t:(r,s,t)\in\TG\}.
\]
Taking $A,B$ with $|AB|=\mu_G(r,s)$ gives
\[
\Mstar(r,s)\le\mu_G(r,s).
\]
The reverse inequality is Proposition~\ref{prop:soundness}. Hence
\[
\mu_G(r,s)=\Mstar(r,s).
\]
\end{proof}

\begin{remark}
Theorem~\ref{thm:exact-region} shows that $\TG$ and $\Mstar$ are intrinsic, although the recursion is evaluated through a H\"older presentation. If $AB=G$, the associated trio has empty third component, which is already covered by the zero-coordinate vectors. In particular, $r+s>|G|$ forces $AB=G$ because $A\cap xB^{-1}\ne\varnothing$ for every $x\in G$.
\end{remark}

\section{Consequences and special cases}\label{sec:special}

We close with consequences of Theorem~\ref{thm:main}, including known special cases, an explicit diagonal formula for faithful split groups, a generating-set obstruction, strict intermediate values between the Eliahou--Kervaire bounds, and a nonsplit example.

\subsection{Previously known classes}

For abelian metacyclic groups, Theorem~\ref{thm:main} gives the same intrinsic function as the formula $\mu_G=\kappa_G$ of \cite{EliahouKervairePlagne2003}. For nilpotent metacyclic groups, normal subgroups occur at every divisor order, so the bounds in Theorem~\ref{thm:ek-bounds} coincide; see \cite[Corollary~9]{BerchenkoKogan2012}. The dihedral formula is \cite[Theorem~7]{EliahouKervaire2010}. For nonabelian groups of order $pq$, with $p>q$ odd primes, \cite[Theorem~20]{BerchenkoKogan2012} gives $\mu_G=\NkappaG(r,s)$ when $r,s>q$ and
\[
\left\lceil\frac rq\right\rceil+
\left\lceil\frac sq\right\rceil<p,
\]
and $\mu_G=\kappa_G$ otherwise. These values are therefore recovered by Theorem~\ref{thm:main}. The theorem also covers the inversion-action H\"older family studied by Benavides and Mutis \cite{BenavidesMutis2025}.

\subsection{Faithful split metacyclic groups}

Consider
\[
G=C_p\rtimes_h C_m,
\qquad
\ord_p(h)=m,
\qquad
m\mid p-1,
\]
where $p$ is prime. In the diagonal nonsaturated range, the general theorem specializes to a particularly simple expression.
\begin{proposition}\label{prop:faithful-split}
Assume $2r-1\le p$. Then
\[
\mu_G(r,r)=
\begin{cases}
\displaystyle \min_{d\mid m}f_d(r,r), & r\le m,\\[1ex]
2r-2, & r>m\text{ and }m,r\text{ are even},\\
2r-1, & r>m\text{ otherwise}.
\end{cases}
\]
Moreover,
\[
\kappa_G(r,r)=\min_{d\mid m}f_d(r,r),
\qquad
\NkappaG(r,r)=2r-1.
\]
\end{proposition}

\begin{proof}
Put $P=C_p$. The subgroup orders are precisely the divisors of $pm$, by
Lemma~\ref{lem:divisors-subgroups}. The normal subgroup orders are $1$ and
$pd$ for $d\mid m$: subgroups containing $P$ are normal because $G/P$ is
cyclic, whereas a normal subgroup $J$ with $J\cap P=1$ satisfies
$[J,P]=1$ and hence $J\le C_G(P)=P$, so $J=1$. Since
$f_{pd}(r,r)\ge p\ge2r-1$,
\[
\DkappaG(r,r)=\kappa_G(r,r)=\min_{d\mid m}f_d(r,r),
\qquad
\NkappaG(r,r)=2r-1.
\]

For $r\le m$, subsets of the cyclic complement attain the first minimum
by the abelian theorem, and the lower bound follows from
Theorem~\ref{thm:ek-bounds}. Suppose henceforth that $r>m$.
Intervals in $P$ give the upper bound $2r-1$. If $m$ and $r$ are even,
let $f=\eta^{m/2}$ and let $I$ be an interval of length $r/2$ in $P$.
Since $h^{m/2}=-1$, the element $f$ acts on $P$ by inversion, and
\[
A=I\{1,f\},\qquad B=\{1,f\}I
\]
satisfy $|A|=|B|=r$ and $|AB|=2r-2$.

Write
\[
A=\bigcup_i A_i\eta^i,
\qquad
B=\bigcup_j B_j\eta^j,
\]
with $A_i,B_j\subseteq\Fp$. Then
\[
(AB)_k=\bigcup_{i+j=k}(A_i+h^iB_j).
\]
Set $a_i=|A_i|$, $b_j=|B_j|$, and
\[
c_k=\max_{\substack{i+j=k\\a_ib_j>0}}(a_i+b_j-1),
\qquad
L=\sum_k c_k.
\]
By Cauchy--Davenport \cite{Cauchy1813,Davenport1935},
$|(AB)_k|\ge c_k$, since $a_i+b_j-1\le2r-1\le p$.
Let $\varepsilon=2$ if $m,r$ are even and $\varepsilon=1$ otherwise.
It remains to consider
\[
L<2r-\varepsilon.
\]

Define
\[
U=\{(i,z):0\le z<a_i\},
\qquad
V=\{(j,z):0\le z<b_j\}
\]
in $C_m\times\Z$. For each $k$, the $k$-fiber of $U+V$ is
\[
\bigcup_{\substack{i+j=k\\a_ib_j>0}}
\{0,\ldots,a_i+b_j-2\}
=
\{0,\ldots,c_k-1\},
\]
hence $|U+V|=L$.

Let $J=\operatorname{Stab}(U+V)$. Since $U+V$ is finite,
$J\le C_m\times\{0\}$; write $|J|=d\mid m$ and $q=m/d$. Put
\[
u=|U+J|-r,
\qquad
v=|V+J|-r.
\]
Kneser's theorem \cite{Kneser1953,Kneser1955}, in the form recalled in
\cite[Theorem~1.2]{DeVos2013}, yields
\[
L\ge2r+u+v-d,
\qquad
u+v<d-\varepsilon.
\]

The $J$-hulls have constant heights on their $q$ cosets. If $R,S$ are
the sums of these heights, then
\[
dR=r+u,\qquad dS=r+v.
\]
Since $0\le u,v<d$ and $u-v=d(R-S)$, we obtain $u=v$ and $R=S$.
Moreover $d\mid L$, and therefore
\[
L=d(2R-1)=2r+2u-d,
\qquad
2u<d-\varepsilon.
\]
The cases $d=1,2$ are impossible, so $d>2$.

As $dR\ge r>m=dq$, the maximal hull heights
\[
\alpha=\max_i a_i,
\qquad
\beta=\max_j b_j
\]
satisfy $\alpha,\beta\ge2$. Choose $J$-cosets on which these maxima
occur, let $X,Y$ be the maximizing indices, and set
\[
x=d-|X|,
\qquad
y=d-|Y|.
\]
Then $x,y\le u$. All $d$ outputs in the sum of these cosets have
profile height
\[
c^\ast=\alpha+\beta-1.
\]
Since $2R\ge\alpha+\beta$ and $2u\le d-\varepsilon-1$,
\[
p\ge2r-1
=2dR-2u-1
\ge d(\alpha+\beta-1)+\varepsilon,
\]
so $c^\ast\le p-2$.

Let $K$ be the set of tight outputs, namely those $k$ with
$|(AB)_k|=c^\ast$, and put $t=|K|$. Join $i\in X$ and $j\in Y$
whenever $i+j\in K$. Every edge gives equality in Cauchy--Davenport,
so by Vosper's theorem \cite{Vosper1956}, the sets
$A_i$, $h^iB_j$, and $(AB)_{i+j}$ are arithmetic progressions with a
common direction. Their directions are unique modulo sign for
cardinalities between $2$ and $p-2$; see
\cite[Lemma~3]{BerchenkoKogan2012}.

Write these directions as
\[
D_i,E_j,F_{i+j}\in\Fp^\ast/\{\pm1\}.
\]
Then
\[
D_i=h^iE_j=F_{i+j},
\qquad
h^{-i}D_i=E_j.
\]
Hence $h^{-i}D_i$ and $E_j$ are constant on each connected component.

Fix a component and an output $k$. If
\[
i=i_0+\ell\frac md
\]
runs through the chosen $J$-coset, then with $g=h^{m/d}$, whose order
is $d$, the equation $F_k=h^iE_j$ reduces in
$\Fp^\ast/\{\pm1\}$ to
\[
[g^\ell]=\text{constant}.
\]
The kernel of
\[
\ell\longmapsto[g^\ell]
\]
has size $1$ for odd $d$ and $2$ for even $d$, hence
\[
e=\gcd(d,2).
\]
Therefore each connected component has at most $et$ edges.

Each vertex of $X$ has degree at least $t-y$, and each vertex of $Y$
has degree at least $t-x$. If $t>\max(x,y)$, then
\[
(t-x)(t-y)\le et.
\]
Since
\[
(t-x)(t-y)-et=t(t-x-y-e)+xy,
\]
we obtain
\[
t\le x+y+e,
\]
with equality only if $xy=0$. The same bound is immediate when
$t\le\max(x,y)$.

If $d$ is odd,
\[
t\le2u+1\le2u+\varepsilon.
\]
If $d$ is even and $\varepsilon=2$, then $t\le2u+2$.
If $d$ is even and $\varepsilon=1$, then $m$ is even and $r$ is odd,
so $u=dR-r$ is positive and odd. Equality $t=2u+2$ would force
$x=y=u>0$, contradicting the equality condition. Thus
\[
t\le2u+\varepsilon.
\]

Exactly $t$ of the selected output fibers are tight; each of the other
$d-t$ has cardinality at least $c^\ast+1$. Hence
\[
|AB|\ge L+d-t
\ge(2r+2u-d)+d-(2u+\varepsilon)
=2r-\varepsilon.
\]
This proves the proposition.
\end{proof}

\subsection{Generating constraints}\label{subsec:generating-constraints}

The function $\mu_G(r,s)$ places no structural restriction on either factor. We compare it with the minimum obtained when the second factor contains the identity, is inverse-closed, and generates $G$; after removing the identity, these are precisely the conditions for an undirected connected Cayley graph. For $1\le s\le |G|$, set
\[
\mathcal B_s(G)
=
\left\{
B\subseteq G:
|B|=s,\;
1\in B,\;
B=B^{-1},\;
\langle B\rangle=G
\right\}.
\]
Whenever $\mathcal B_s(G)\neq\varnothing$, define
\[
\mugen(r,s)
=
\min_{\substack{A\subseteq G,\ |A|=r\\ B\in\mathcal B_s(G)}}|AB|.
\]
Then
\[
\mu_G(r,s)\le \mugen(r,s).
\]
For faithful split metacyclic groups, the inequality is strict in the following diagonal range.

\begin{proposition}[Generating obstruction]\label{prop:generating-obstruction}
Let
\[
G=C_p\rtimes_h C_m,
\qquad
\ord_p(h)=m>2.
\]
Suppose that
\[
m<r,
\qquad
2r-1\le p.
\]
If $A,B\subseteq G$ satisfy
\[
|A|=|B|=r,
\qquad
1\in B,
\qquad
B=B^{-1},
\qquad
\langle B\rangle=G,
\]
then
\[
|AB|\ge 2r.
\]
Equivalently, whenever $\mathcal B_r(G)\neq\varnothing$,
\[
\mugen(r,r)\ge 2r.
\]
\end{proposition}

\begin{proof}
Suppose, to the contrary, that
\[
|AB|<2r,
\]
and put
\[
\delta=2r-|AB|>0.
\]
Fix $x\in A$ and let $H\le G$ be the unique minimal subgroup for which
\[
A\subseteq xH.
\]
Choose $y\in B$. DeVos's Corollary~14.4 \cite{DeVos2013} gives $B\subseteq Hy$ and one of four structural alternatives. Since $1\in B$, we have $y\in H$, hence $Hy=H$ and $B\subseteq H$. Because $B$ generates $G$,
\[
H=G.
\]

Write
\[
P=C_p\triangleleft G.
\]
We first record two elementary consequences of the faithful semidirect-product structure.

Every proper subgroup $K<G$ satisfies
\[
|K|\le\frac{|G|}{2}=\frac{pm}{2}.
\]
Indeed, if $p\nmid |K|$, then $|K|\mid m$, hence $|K|\le m\le pm/2$. If $p\mid |K|$, then the unique Sylow $p$-subgroup $P$ is contained in $K$, and
\[
|K|=pd
\]
for a proper divisor $d$ of $m$; hence $d\le m/2$.

Moreover, every nontrivial normal subgroup of $G$ contains $P$. To see this, let $K\triangleleft G$ and suppose $K\cap P=1$. Then
\[
[K,P]\le K\cap P=1,
\]
so $K\le C_G(P)$. Faithfulness of the action gives
\[
C_G(P)=P,
\]
and consequently $K=1$. Thus, if $1\ne K\triangleleft G$, then $P\le K$, and in particular $G/K$ is cyclic.

We exclude the four alternatives of Corollary~14.4 in turn.

In the first alternative, either $AB=G$, which is impossible because
\[
|AB|\le 2r-1\le p<pm=|G|,
\]
or there exist a proper subgroup $K<G$ and $z\in G$ such that
\[
x(G\setminus zK)y\subseteq AB.
\]
Hence
\[
|AB|
\ge |G|-|K|
\ge \frac{pm}{2}
>p,
\]
again contradicting $|AB|\le p$.

In the second alternative, there exists $K\triangleleft G$ such that $G/K$ is cyclic of order at least four. Since $G$ is nonabelian, $K\ne1$, and therefore
\[
P\le K,
\qquad
|K|\ge p>r.
\]
Since $B$ generates $G$, it is not contained in $K$. Because $1\in B$, the set $KB$ contains $K$ and at least one further left $K$-coset. Consequently,
\[
|KB\setminus B|
=
|KB|-|B|
\ge 2|K|-r
>|K|.
\]
This contradicts both possible inequalities in the second alternative of Corollary~14.4.

In the third alternative, there exists $K\triangleleft G$ such that $G/K$ is a nonabelian dihedral group. If $K\ne1$, then $P\le K$ and $G/K$ is cyclic, a contradiction. If $K=1$, then $G$ itself would be dihedral. This is impossible when $m>2$. If $m$ is odd, then $|G|=pm$ is odd. If $m$ is even and $G$ were dihedral of order $pm$, every element of order greater than two would lie in the cyclic rotation subgroup of order $pm/2$. But $G$ contains the element $\eta$ of order $m>2$, while
\[
m\nmid \frac{pm}{2}
\]
because $p$ is odd.

It remains to exclude the fourth alternative. Let
\[
d=|K_2|.
\]
Corollary~14.4 gives
\[
|AK_2|=2d=|K_2B|
\]
and
\[
\delta
+
|AK_2\setminus A|
+
|K_2B\setminus B|
=
|K_1|
<d.
\]
Since $A\subseteq AK_2$ and $B\subseteq K_2B$,
\[
\delta+4d-2r<d.
\]
In particular,
\[
r\le2d
\]
and
\[
3d<2r-\delta<2r\le p+1.
\]
Hence $d<p$. Since $d\mid pm$, it follows that $d\mid m$. If $d$ were a proper divisor of $m$, then $d\le m/2$, and therefore
\[
r\le2d\le m,
\]
contrary to $r>m$. Thus
\[
d=m.
\]

Hence $K_2$ is a complement to $P$. Every such complement is $P$-conjugate to the standard complement
\[
Q=\langle\eta\rangle\cong C_m.
\]
Indeed, the projection $K_2\to G/P$ is an isomorphism, so $K_2$ has a generator of the form $\tau^a\eta^u$ with $\gcd(u,m)=1$. Since $h^u\ne1$, conjugation by a suitable $\tau^t$ sends this generator to $\eta^u$. Conjugating the entire configuration preserves all the assumptions, so we may assume
\[
K_2=Q.
\]

Now
\[
|QB|=2m.
\]
Since $1\in B$, one of the two left $Q$-cosets in $QB$ is $Q$ itself. Thus, for some $y\notin Q$,
\[
B\subseteq Q\cup Qy.
\]
Let $b\in B\cap Qy$. Since $B=B^{-1}$ and $b\notin Q$, we also have
\[
b^{-1}\in B\cap Qy.
\]
Thus $Qb=Qb^{-1}$, which implies
\[
b^2\in Q.
\]

Write
\[
b=\tau^a\eta^i.
\]
Since $b\notin Q$, one has $a\ne0$, and
\[
b^2
=
\tau^{a(1+h^i)}\eta^{2i}.
\]
Therefore $b^2\in Q$ forces
\[
h^i=-1.
\]
There is at most one residue class $i\pmod m$ with this property. On the other hand, a fixed left coset $Qy$ contains exactly one element with each possible image in $G/P\cong C_m$. Hence
\[
|B\cap Qy|\le1.
\]
It follows that
\[
r=|B|\le m+1.
\]
Since $r>m$, necessarily $r=m+1$. Substituting $d=m$ and $r=m+1$ into the inequality obtained from the fourth alternative gives
\[
\delta+2m-2<m,
\]
which is impossible because $\delta\ge1$ and $m>2$.

All four alternatives are impossible, so
\[
|AB|\ge2r.
\]
\end{proof}

\begin{corollary}[Strict generating gap]\label{cor:generating-gap}
Under the hypotheses of Proposition~\ref{prop:generating-obstruction}, assume that $\mathcal B_r(G)\neq\varnothing$. Then
\[
\mugen(r,r)>\mu_G(r,r).
\]
More precisely,
\[
\mugen(r,r)-\mu_G(r,r)
\ge
\begin{cases}
2, & m\text{ and }r\text{ are even},\\[1mm]
1, & \text{otherwise}.
\end{cases}
\]
In particular, no pair attaining the unrestricted minimum $\mu_G(r,r)$ can have its second factor symmetric, containing the identity, and generating $G$.
\end{corollary}

\begin{proof}
Proposition~\ref{prop:generating-obstruction} gives
\[
\mugen(r,r)\ge2r.
\]
Since $r>m$, Proposition~\ref{prop:faithful-split} gives
\[
\mu_G(r,r)=
\begin{cases}
2r-2, & m\text{ and }r\text{ are even},\\
2r-1, & \text{otherwise}.
\end{cases}
\]
The result follows.
\end{proof}

\begin{remark}
The condition $\mathcal B_r(G)\ne\varnothing$ is nonvacuous whenever parity permits. For odd $r$, start from $\{1,\eta,\eta^{-1},\tau,\tau^{-1}\}$ and add inverse pairs from $P$. If $r$ is even and $m$ is even, also include the involution $\eta^{m/2}$. The bound $2r-1\le p$ leaves enough inverse pairs. If $m$ is odd, then $|G|$ is odd and every inverse-closed subset containing $1$ has odd cardinality.
\end{remark}

\begin{remark}[Cayley interpretation]\label{rem:cayley-interpretation}
If $B\in\mathcal B_s(G)$ and
\[
S=B\setminus\{1\},
\]
then $S=S^{-1}$, $\langle S\rangle=G$, and $\operatorname{Cay}(G,S)$ is connected. Moreover, for every $A\subseteq G$,
\[
|AB|-|A|
=
|AS\setminus A|.
\]
Consequently,
\[
\mugen(r,s)-r
=
\min_{\substack{
S=S^{-1},\ 1\notin S\\
|S|=s-1,\ \langle S\rangle=G
}}
\;
\min_{|A|=r}|AS\setminus A|.
\]
Thus $\mugen(r,s)-r$ is the lower envelope, at cardinality $r$, of the outer vertex-isoperimetric profiles of connected Cayley graphs of $G$ with symmetric generating sets of size $s-1$. This places \(\mu_G^{\mathrm{gen}}\) within the isoperimetric framework for additive and product-set problems developed by Hamidoune \cite{Hamidoune1996,Hamidoune1999}.
\end{remark}

\subsection{Intermediate values}

The two classical arithmetic bounds do not exhaust the possible values of the minimum product-set function.

\begin{corollary}\label{cor:intermediate}
Let $m>2$ be even, let $r=km$ with $k\ge2$, and let $p\equiv1\pmod m$ be prime with $2r-1\le p$. For
\[
G=C_p\rtimes_h C_m,
\qquad
\ord_p(h)=m,
\]
one has
\[
\DkappaG(r,r)=\kappa_G(r,r)=2r-m
<\mu_G(r,r)=2r-2
<\NkappaG(r,r)=2r-1.
\]
For each fixed $m$ and $k$ satisfying these assumptions, this gives infinitely many groups.
\end{corollary}

\begin{proof}
Every $d\mid m$ divides $r$, so $f_d(r,r)=2r-d$. Proposition~\ref{prop:faithful-split} gives the displayed equalities, and $m>2$ gives strictness. Dirichlet's theorem on primes in arithmetic progressions gives infinitely many primes $p\equiv1\pmod m$; discarding finitely many of them ensures $2r-1\le p$.
\end{proof}

For example, the values present in Table \ref{tab:intermediate} follow directly from Corollary~\ref{cor:intermediate}; none of these requires a numerical search.

\begin{table}[htpb]
\caption{Examples of strict intermediate values between the Eliahou--Kervaire bounds.}\label{tab:intermediate}
\centering
\begin{tabular}{@{}ccccc@{}}
\toprule
$G$ & $r=s$ & $\DkappaG(r,r)=\kappa_G(r,r)$ & $\mu_G(r,r)$ & $\NkappaG(r,r)$\\
\midrule
$C_{17}\rtimes C_4$ & $8$  & $12$ & $14$ & $15$\\
$C_{31}\rtimes C_6$ & $12$ & $18$ & $22$ & $23$\\
$C_{37}\rtimes C_6$ & $18$ & $30$ & $34$ & $35$\\
\bottomrule
\end{tabular}
\end{table}

In each row the action is faithful. The condition $2r-1\le p$ holds in all three cases.

\subsection{A nonsplit example}

Consider $K=K(9,57,19,4)$ in Yang's convention. Here
\[
4^9\equiv1\pmod{57},
\qquad
19(4-1)\equiv0\pmod{57},
\]
so $|K|=513$. Since $\eta^9=\tau^{19}$ has order three and the image of $\eta$ in $K/\angles{\tau}$ has order nine, $\eta$ has order $27$. Consequently
\[
H=\angles{\eta^3}
\]
has order nine, and
\[
H\cap\angles{\tau}=\angles{\eta^9}=\angles{\tau^{19}}
\]
has order three. The subgroup
\[
N=\angles{\tau,\eta^3}
\]
is the preimage of the order-three subgroup of $K/\angles{\tau}\cong C_9$, hence is normal of order $171$ and index three. The corresponding extension $1\to\angles{\tau}\to K\to C_9\to1$ is nonsplit. Indeed,
\[
\sum_{j=0}^{8}4^j
\equiv(1+4+16)(1+7+49)
\equiv0\pmod{57},
\]
so
\[
(\tau^b\eta)^9=\tau^{19}\ne1
\]
for every $b$. A complement of order nine would contain an element mapping to $\eta\angles{\tau}$ whose ninth power is one, a contradiction. Thus the H\"older extension with kernel $\angles{\tau}$ is nonsplit. The same abstract group also has the split presentation
\[
C_{19}\rtimes_4 C_{27},
\]
generated by $\tau^3$ and $\eta$: their cyclic subgroups have coprime orders $19$ and $27$ and together generate $K$ because $\eta^9=\tau^{19}$.

Take
\[
A=B=N\cup\eta H.
\]
The two cosets of $N$ involved are distinct, hence
\[
|A|=|B|=171+9=180.
\]
Normality of $N$, the inclusion $H\le N$, and $\eta H=H\eta$ give
\[
AB=N\cup\eta N\cup\eta^2H,
\qquad
|AB|=171+171+9=351.
\]
These three pieces are disjoint, as their images in $K/N$ are distinct. To certify the lower bound independently of this construction, the divisors of $513$ and their associated bounds are shown in Table~\ref{tab:nonsplit}.

\begin{table}[htpb]
\caption{Divisor bounds for the nonsplit H\"older presentation $K(9,57,19,4)$.}\label{tab:nonsplit}
\centering
\begin{tabular}{@{}c*{8}{r}@{}}
\toprule
$d$ & 1 & 3 & 9 & 19 & 27 & 57 & 171 & 513\\
\midrule
$f_d(180,180)$ & 359 & 357 & 351 & 361 & 351 & 399 & 513 & 513\\
\bottomrule
\end{tabular}
\end{table}

Theorem~\ref{thm:ek-bounds} therefore gives
\[
\mu_K(180,180)\ge D_{\kappa_K}(180,180)=351.
\]
Together with the explicit product above, this proves
\[
\mu_K(180,180)=351
\]
and hence, by Theorem~\ref{thm:main},
\[
M_K^{\ast}(180,180)=351.
\]
Thus the displayed construction attains the exact minimum.

\begin{remark}
Exhaustive comparisons with direct product-set enumeration were carried out for all valid Hölder presentations of order at most \(12\). The dihedral boundary constructions \(W0,W1,W2A,W2B\) were additionally checked through \(19\,202\) independent concrete lifts, including split and nonsplit extensions and examples with nonabelian kernel.
\end{remark}


\section*{Statements and Declarations}

\noindent\textbf{Funding.}
No funding was received to assist with the preparation of this manuscript.

\medskip
\noindent\textbf{Competing Interests.}
The author declares no competing interests.

\medskip
\noindent\textbf{Data Availability.}
No datasets were generated or analyzed in this study.

\medskip
\noindent\textbf{Code Availability.}
Code and reproducibility materials for the computational checks are available at the \href{https://github.com/Santymax98/papers/tree/main/Jimenez_2026_Product_sets}{project repository}.

\medskip
\noindent\textbf{Author Contributions.}
Santiago Jimenez Ramos: Conceptualization, Methodology, Formal analysis, Investigation, Software, Validation, Writing -- original draft, Writing -- review and editing.

\bibliographystyle{unsrt}
\bibliography{references}

\end{document}